\documentclass{amsart}
\usepackage{amssymb}
\usepackage{hyperref}
\usepackage{mathrsfs}
\usepackage{mathtools}
\usepackage{centernot}
\usepackage{bbm}
\usepackage{dirtytalk}

\usepackage[utf8]{inputenc}
\newcommand{\Z}{\mathbb{Z}}	
\newcommand{\R}{\mathbb{R}}	
\newcommand{\N}{\mathbb{N}}	
\newcommand{\T}{\mathbb{T}}

\renewcommand{\vec}[1]{\boldsymbol{#1}}

\newcommand{\paren}[1]{\left( #1 \right)}
\newcommand{\brac}[1]{\left[ #1 \right]}

\newcommand{\set}[1]{\left\{#1\right\}}

\newcommand{\interior}[1]{%
  {\kern0pt#1}^{\mathrm{o}}%
}

\newtheorem{theorem}{Theorem}
\newtheorem{lemma}[theorem]{Lemma}
\newtheorem{prop}[theorem]{Proposition}

\newtheorem{question}[theorem]{Question}
\newtheorem{conjecture}[theorem]{Conjecture}
\newtheorem{corollary}[theorem]{Corollary}

\title{On the Continuity of Enhancement Percolation}
\author{Paul Duncan}
\author{Benjamin Schweinhart}
\author{David Sivakoff}

\date{\today}

\begin{document}
\begin{abstract}
    We study bond percolation in $\Z^d$ with an unbounded family of enhancements that enable additional bonds to act as open. A natural question is whether percolation occurs in this model if and only if percolation also occurs in the system with a finite subcollection of enhancements. We give an affirmative answer in dimension $d=2$ for symmetric families of connected enhancements, and in dimensions $d\ge 3$ we prove a partial result.
\end{abstract}

\maketitle

\section{Introduction}
A natural question in the study of long-range percolation is whether a supercritical system remains supercritical if edges above a certain fixed length are removed. This truncation phenomenon was originally considered in~\cite{meester1996continuity} and has since been an active area of study~\cite{menshikov2001note,van2016truncated,baumler2024truncation}. We address the analogous question for enhancement percolation wherein additional connections are included based on the presence of patterns in an underlying bond percolation model. The prototypical example of enhancement percolation is entanglement percolation on $\Z^3.$ Consider bond percolation $P$ with edge probability $p$ on $\Z^3$ and declare two vertices to be connected if they are entangled in the sense that there is no dual sphere separating them in $\R^3\setminus P.$ It was shown in \cite{AG91,balister2014essential} that the critical probability $p_c^{\mathrm{ent}}$ for entanglement percolation on $\Z^3$ is strictly less than the critical probability for bond percolation $p_c\paren{\Z^3}.$ 

One formulation of entanglement percolation can be rephrased as a long-range percolation model. A finite graph $G\subset \Z^3$ is entangled if there is no sphere in $\R^3\setminus G$ separating any of the vertices of $G$ (for example, take $G$ to be a Hopf link). Denote by $\mathscr{F}_k$  the set of entangled subgraphs of $\Z^3$ of $\ell^{\infty}$ diameter between $2^k$ and $2^{k+1}.$ Define level $k$ entanglement percolation $P_k$ to be the graph obtained from $P$ by adding connections between vertices contained in a graph of $\bigcup_{j=0}^k\mathscr{F}_k.$ It is natural to ask whether the critical probabilities $p_k^{\mathrm{ent}}$ for the percolations $P_k$ are continuous in the sense that
$$\lim_{k\to\infty} p_k^{\mathrm{ent}} = p_c^{\mathrm{ent}}\,.$$
In a separate paper, we prove a number of results about entanglement percolation conditional on an affirmative answer to a related question~\cite{duncan2025dual}.

The general continuity question for enhancement percolation is as follows. An enhancement $E$ is a subgraph $T$ of the nearest-neighbor graph $\Z^d$ together with an enhanced subgraph $S$ of the complete graph on $\Z^d.$ For convenience we assume that $T\subset S.$ Enhancement percolation with probability $p$ on $\Z^d$ --- denoted $P_{\infty}=P_{\infty}\paren{p}$ --- for a set of enhancements $\mathscr{E}=\set{\paren{T_\alpha,S_\alpha}}_{\alpha\in I}$ is the graph obtained by taking bond percolation $P=P\paren{p}$ with probability $p$ and adding an appropriately transformed copy of $S_\alpha$ for every subgraph of $P$ congruent to $T_\alpha$ under the action of $SO(d)$ (note: while we assume that the family of enhancements is symmetric, anisotropic enhancement percolation may also be of interest). As before, let $\mathscr{F}_k$ be the collection of enhancements $E=\paren{T,S}$ so that the $l^{\infty}$ diameter of $S$ in $\Z^d$ is between $2^k$ and $2^{k+1}.$ Define $P_k$ to be the graph obtained from $P$ by activating the enhancements in $\bigcup_{j=1}^k\mathscr{F}_j.$  

\begin{question}\label{question:cont}
Let $p_c^e$ and $p_k^e$ denote the critical probabilities for $P_{\infty}$ and $P_k,$ respectively. Under what hypotheses do we have that
$$\lim_{k\to\infty} p_k^e=p_c^e\,?$$
\end{question}

Our main result answers this question for a special class of enhancement percolations on $\Z^2.$

\begin{theorem}\label{theorem:z2continuity}
Let $\mathscr{E}=\set{\paren{T_\alpha,S_\alpha}}_{\alpha\in I}$ be a set of enhancements for bond percolation on $\Z^2$ so that $S_{\alpha}$ is a connected subgraph of the nearest-neighbor graph $\Z^2$ for each $\alpha\in I.$ Then
$$\lim_{k\to\infty} p_k^e=p_c^e\,.$$
\end{theorem}

The hypotheses may be relaxed slightly: the same arguments work for classes of enhancements where we do not require $T_{\alpha}\subset S_{\alpha}$ but instead that there is a $c>0$ so that $\mathrm{diam}\, T_{\alpha}\cup S_{\alpha}\leq c \,\mathrm{diam}\, S_{\alpha}$ for all $\alpha \in I.$ In particular, the theorem holds for enhancements of the type depicted in Figure~\ref{fig:junta} or (say) one where a crossing of a box is added if the density of edges in the box exceeds a certain threshold.

\begin{figure}[t]
    \centering
    \includegraphics[width=.6\textwidth]{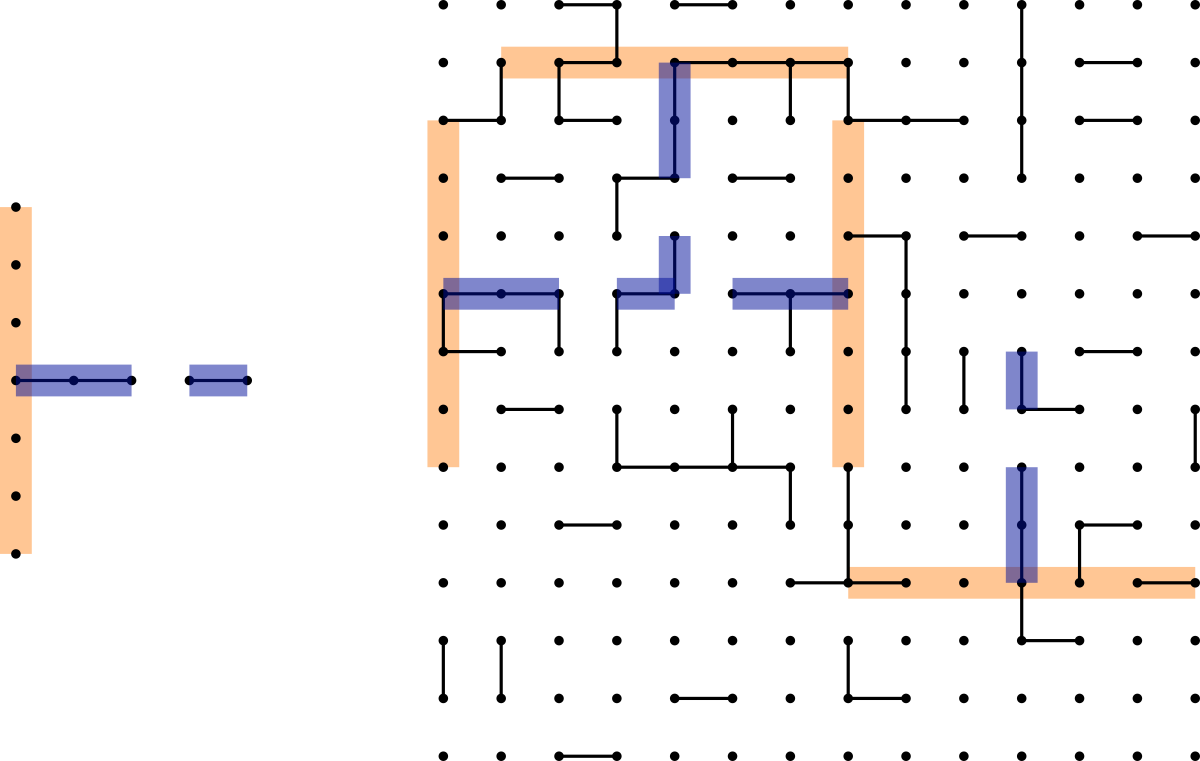}
    \caption{An illustration of enhancement percolation with a single enhancement. $T$ is shown in blue and $S\setminus T$ is shown in orange. For every subgraph of Bernoulli percolation (shown in black) congruent to the blue graph an appropriately transformed collection of (lighter) orange edges is added. Here, a horizontal crossing occurs in the enhanced percolation but not in the original percolation.} 
    \label{fig:junta}
\end{figure}

A second question concerns the sharpness of the enhancement percolation transition. Let $C$ be the component of the origin in $P_{\infty}.$  
Set
\[\pi_c^e = \inf\{p:\mathbb{E}(|C|) = \infty\}\,.\]
 \begin{question}
 Under what hypotheses do we have that
$$p_c^e=\pi_c^{e}?$$
\end{question}
While this question is interesting, we do not present any results towards its resolution. 

\subsection{Notation}

We establish some notation. Let $\gamma\in SO(d)$ be a transformation and $\mathscr{F}\subset \mathscr{E}$ be a collection of enhancements. We say that $\paren{\gamma\paren{T},\gamma\paren{S}}$ is an activated enhancement of $\mathscr{F}$ with anchor at $\gamma((0,0))$ if $\gamma(T)\subset P$ for some enhancement $(T,S)\in \mathscr{F}$. We assume that every $\paren{T_{\alpha},S_{\alpha}}$ in $\mathscr{E}$ appears at most once up to congruence and that $\paren{0,0}$ is one of the vertices of $S_{\alpha}.$  Write 
\[\Lambda_n \coloneqq \brac{-n,n}^d\]
for the hypercube of side length $2n$ and 
\[\Lambda_{n_1,n_2,\ldots,n_d} \coloneqq \brac{-n_1,n_1} \times \brac{-n_2,n_2} \times \ldots \times \brac{-n_d,n_d}\]
for a box with the side lengths $2n_1,\ldots,2n_d.$ Define the radius of an enhancement $\paren{T_{\alpha},S_{\alpha}}$ as the smallest $n$ for which $S_{\alpha} \subset \Lambda_n.$ 

For a rectangle $R,$ let $H\paren{R}$ be the event that there is a horizontal crossing of $R,$ that is, a path of open edges in $R$ between the two faces of $R$ that are orthogonal to the first coordinate direction vector $\vec{e}_1.$ Likewise, let $V\paren{R}$ be the event that there is a vertical crossing, or a path between the faces of $R$ that are orthogonal to $\vec{e}_2.$

\section{Continuity for Enhancements in the Plane}

Rather than working with the usual supercritical and subcritical regimes directly, we instead define a new critical probability related to the presence of arbitrarily large enhancements near the origin. Let $\mathscr{E}=\set{\mathcal{E}_j}_{j\in \N}$ be a family of enhancements. 

Let $G_k$ be the event that the anchor of an activated enhancement of $\mathscr{F}_k$ is contained in $\Lambda_{2^{k}}.$ Set
$$q_c = q_c\paren{\mathscr{E}} \coloneqq \sup\set{p:\limsup \mathbb{P}_p\paren{G_k}=0}\,.$$
Note that we would obtain the same threshold if we modify the definition of $G_k$ by replacing $2^{k}$ with $c2^{k}$ for some $c>0.$  Theorem~\ref{theorem:z2continuity} is a consequence of the following result.

\begin{theorem}\label{theorem:qccontinuity}
Assume the same hypotheses as Theorem~\ref{theorem:z2continuity}. If $p<q_c$ then $P_{\infty}$ percolates if and only if $P_{N}$ percolates for some finite $N.$ On the other hand, if $p>q_c$ then $P_N$ percolates for some finite value of $N.$ 
\end{theorem}

This is not a usual percolation threshold in that it does not necessarily correspond to an obvious phase transition. It is instead intended to capture a change in the locality of enhancement behavior. Both above and below $q_c,$ our main tools are sharp threshold results.

\subsection{Sharpness}
For $q<q_c,$ we want to show that there is a $N$ for which 
\[\mathbb{P}_q\paren{\bigcup_{k \geq N} G_k} < 1.\]
Both the statement and the proof of the  following theorem of Friedgut and Kalai~\cite{friedgut1996every} will be useful. 

\begin{theorem}[Friedgut, Kalai]\label{theorem:FK}
    There is a universal constant $C$ so that the following holds: Let $A$ be a symmetric, increasing event on the Hamming cube $Q_n = \set{0,1}^n$ endowed with the Bernoulli product measures. If $\mathbb{P}_p\paren{A} > \epsilon,$ then 
    \[\mathbb{P}_q\paren{A} > 1-\epsilon\]
    for $q = p + C\log\paren{1/2\epsilon}/\log n\,.$
\end{theorem}

To show that $G_k$ has a sufficiently sharp threshold, we compare it to a similar event on a torus in order to apply the Friedgut-Kalai machinery for symmetric events. Let $\mathbb{T}_m$ be the torus of diameter $m$ obtained by identifying opposite faces of $\Lambda_{m}.$ An enhancement is well-defined in the torus $\mathbb{T}_m$ when its radius is at most $m$. 

\begin{lemma}\label{lemma:symcomp} Let $l \geq 3$ and let $J_{k,l}$ be the event that there is an activated enhancement of $\mathscr{F}_k$ in $\mathbb{T}_{l*2^k}.$ Then
    \[\mathbb{P}\paren{G_k}  \leq \mathbb{P}\paren{J_{k,l}} \leq 1- \paren{1-\mathbb{P}\paren{G_k}}^{l^d}\,.\]
\end{lemma}

\begin{proof}
    The first inequality is immediate, since $G_k$ only depends on states in $\Lambda_{2^{k+1}},$ which embeds into $\mathbb{T}_{l*2^k}.$ To see the second inequality, notice that $J_{k,l}$ can be written as a union of $l^d$ translates of $G_k,$ which we enumerate as $\set{G^{\paren{i}}}_{i=1}^{l^d}.$ Then using FKG, we have
    \[\mathbb{P}\paren{\neg J_{k,l}} = \mathbb{P}\paren{\bigcap_i \neg G^{\paren{i}}} \geq \mathbb{P}\paren{\neg G_k}^{\paren{l^d}}\,,\]
    so
    \[\mathbb{P}\paren{J_{k,l}} \leq 1- \paren{1-\mathbb{P}\paren{G_k}}^{l^d}\,\]
    as desired.
\end{proof}

We now consider the sharpness of the probability of a symmetric event in the torus.

\begin{prop}\label{prop:FKsummable}
    Consider independent Bernoulli bond percolation on the sequence of tori $\set{\T_{2^{k}}}_{k \in \N}.$ For each $k \in \N,$ let $A_k$ be a symmetric increasing event on $\T_{2^{k}},$ and let 
    \[r_c \coloneqq \sup\set{p : \limsup \mathbb{P}_p\paren{A_k} = 0}\,.\]
    Then if $p < r_c,$
    \[\sum_{k=1}^{\infty} \mathbb{P}_p\paren{A_k} < \infty\,,\]
    and if $p > r_c,$
    \[\limsup_{k \to \infty} \mathbb{P}_p\paren{A_k} = 1\,.\]
\end{prop}

\begin{proof}
The second claim is a straightforward consequence of Theorem~\ref{theorem:FK}. We now turn to the first claim. For convenience, let $r_k$ satisfy $\mathbb{P}_{r_k}\paren{A_k} = 1/2$ for each $k.$ Recall that the proof~\cite{friedgut1996every} of Theorem~\ref{theorem:FK} involves the following inequality:
$$\frac{\partial\mathbb{P}_{p}\paren{A_k}}{\partial p}\geq c_3 \mathbb{P}_{p}\paren{A_k} \log\paren{2^{k+1}}\,,$$
which holds for all $p\leq r_k,$ where $c_3$ is a universal constant that does not depend on the family of events $A_k.$  
Therefore, if $f_k\paren{q}$ satisfies
$$\frac{\partial f_k\paren{p}}{\partial p}= c_3 f_k\paren{p} \log\paren{2^{k+1}}, \quad f_k\paren{r_k}=1/2,$$
then $\mathbb{P}_{p}\paren{A_k}\leq f_k\paren{p}$ for $p\leq r_k.$ Set $c_4=c_3 d.$ We have that 
$$f_k\paren{p}=A e^{c_4 \paren{k+1} \, p \log 2}= 2^{c_4 \paren{k+1} \paren{p-r_k}-1}$$
which is a convergent geometric series for $p<\liminf_k r_k.$ Thus, since $r_c \leq \liminf_k r_k,$ if $p<r_c$ we have $\sum_{k=1}^\infty \mathbb{P}_{p}\paren{A_k} < \infty.$
\end{proof}

\begin{corollary}\label{cor:summable}
    For any $\epsilon > 0,$ 
    \[\sum_{k=1}^\infty \mathbb{P}_{q_c-\epsilon}\paren{G_k} < \infty\]
    and
    \[\limsup_{k \to \infty} \mathbb{P}_{q_c + \epsilon}\paren{G_k} = 1\,.\]
\end{corollary}

\begin{proof}
    By Lemma~\ref{lemma:symcomp}, $G_k$ has a sharp threshold if and only if $T_{k,3}$ does, so Proposition~\ref{prop:FKsummable} gives the desired bounds.
\end{proof}

\subsection{Continuity below $q_c$}

We are now ready to work towards proving Theorem~\ref{theorem:qccontinuity}.

First, we consider the case $p < q_c.$ 
Since the probabilities of the events $G_k$ are summable in this regime by Corollary~\ref{cor:summable}, there will not be percolation in $P_{\infty}$ unless $P_{k}$ contains paths of length $2^{k}$ starting from the origin with constant probability for arbitrarily large $k.$ We will then show that these paths are enough to percolate in the plane when $k$ is large enough.

\begin{lemma}\label{lemma:onearm}
    Let $\mathbb{P}$ be the law of a translation invariant positively associated subgraph of $\Z^2.$ Then for any $j \in \N,$
    \[\mathbb{P}\paren{H\paren{\Lambda_{jk,2k}}} \geq \paren{\frac{1}{8}\mathbb{P}\paren{0 \xleftrightarrow[]{} \partial \Lambda_k}}^{2j}\,.\]
\end{lemma}

\begin{proof}
    This follows from the one-arm estimate by exploring to one half of the right side of a square, choosing the upper or lower part as necessary to remain in the prescribed rectangle.
\end{proof}

Since we will consider multiple percolation sets simultaneously in the following lemma, we introduce a modification of the box crossing notation in order to avoid confusion. Define $H\paren{R,P}$ to be the event that there is a horizontal crossing of $R$ in the random subgraph $P.$

\begin{lemma}\label{lemma:torussharp}
    Suppose that there is an $\epsilon>0$ and a $p \in \brac{0,1}$ so that for every $k,$ 
    \[\mathbb{P}_p\paren{H\paren{\Lambda_{7*2^k,2^k},P_k}} > \epsilon\,.\]
    Then for any $p' > p,$ there is an $M = M\paren{p'}$ so that for $k>M,$
    \[\mathbb{P}_{p'}\paren{H\paren{\Lambda_{6*2^k,2*2^k},P_k}} > 1- \epsilon\,.\]
\end{lemma}

\begin{proof}
    The proof is essentially the same as the proof of Lemma 9 of~\cite{bollobas2006short}.
\end{proof}

Recall that a site percolation in $\Z^d$ is called $k$-dependent if the state of a vertex $v$ is independent of the states outside of $\Lambda_k\paren{v},$ the cube of radius $k$ centered at $v.$ We use the following standard tool for comparing $k$-dependent percolations to independent ones.
\begin{theorem}[Liggett, Schonmann, Stacey]
    For any $d,k \in \N, \epsilon > 0$ there is a $\delta > 0$ so that if $X$ is a $k$-dependent site percolation on $\Z^d$ that satisfies
    \[\mathbb{P}\paren{v \in X \mid X \cap \paren{\Lambda_k\paren{v} \setminus v} = \omega} > 1-\delta\]
    for any $v \in \Z^d, \omega \subset \Z^d,$ then $X$ stochastically dominates a Bernoulli site percolation with parameter $1-\epsilon.$
\end{theorem}

When we apply the Lemma~\ref{lemma:torussharp} in the next proof note that the dependence of enhancement activation events is controlled by the bound on the diameter of the enhanced set and the requirement that the enhanced set contains the activation set. 

\begin{proof}[Proof of Theorem~\ref{theorem:qccontinuity} in the case $p < q_c$]
    Suppose that $P_{\infty}$ percolates. By Corollary~\ref{cor:summable}, we have 
    \[\mathbb{P}_p\paren{\bigcup_{k=N}^{\infty} G_k} \xrightarrow[]{N \to \infty} 0\,.\]
    Combining this with the bound
    \begin{align*}
        \mathbb{P}_p\paren{0 \xleftrightarrow[P_{\infty}]{} \infty} &\leq \mathbb{P}_p\paren{0 \xleftrightarrow[P_{N-1}]{} \partial \Lambda_{2^{N}} \cup \bigcup_{k=N}^\infty G_k}\\ &\leq \mathbb{P}_p\paren{0 \xleftrightarrow[P_{N-1}]{} \partial \Lambda_{2^N}} + \mathbb{P}_p\paren{\bigcup_{k=N}^{\infty} G_k}\,,
    \end{align*}
    there must exist an $\epsilon > 0$ and an $N \in \N$ so that for any $k > N,$ 
    \[\mathbb{P}_p\paren{0 \xleftrightarrow[P_{k-1}]{} \partial \Lambda_{2^{k}}} > \epsilon\,.\]

    Now let $p' > p$ and set $0 < \delta < \epsilon$ to be determined later. By Lemmas~\ref{lemma:onearm} and~\ref{lemma:torussharp}, for sufficiently large $k,$ we have that
    \begin{equation}\label{eq:boxcrossing}
    \mathbb{P}_{p'}\paren{H\paren{\Lambda_{6*2^{k},2*2^{k}},P_{k-1}}} > 1- \delta\,.
    \end{equation}
    Now consider a renormalized site system in which the sites are elements of a tiling by translates of $\Lambda_{2*2^{k}}.$ Given such a site $\mathbb{S},$ let $R_1^{\mathbb{S}}$ and $R_2^{\mathbb{S}}$ respectively be the translates of $\Lambda_{6*2^{k},2*2^{k}}$ and $\Lambda_{2*2^{k},6*2^{k}}$ centered at the center point of $\mathbb{S}.$ 
    Then we say that $\mathbb{S}$ is open if $H\paren{R_1^{\mathbb{S}}, P_{k-1}}$ and $V\paren{R_2^{\mathbb{S}},P_{k-1}}$ both occur. Then by positive association and the previous estimate we have
    \[\mathbb{P}_{p'}\paren{\text{$\mathbb{S}$ is open}} \geq \paren{1-\delta}^2\,,\]
    and if two renormalized vertices are connected by a path of open sites, they are connected by a path in $P_{k-1}.$ Notice that the renormalized system is $2$-dependent. Thus, by~\cite{liggett1997domination}, it is supercritical for sufficiently small $\delta.$ We have therefore shown that $P_{k-1}$ percolates.
\end{proof}

\subsection{Continuity above $q_c$}

We now consider the case of $p>q_c,$ where we will prove that percolation always occurs with finitely many enhancements. First, we show that activated connected enhancements produce box crossings in the plane.

\begin{lemma}\label{lemma:occupancycrossing}
    Let $\mathscr{E}$ be a family of enhancements that satisfies the hypotheses of Theorem~\ref{theorem:z2continuity}, each with radius at least $n.$ Let $L_n$ be the event that there is an activated enhancement of $\mathscr{E}$ anchored in $\Lambda_{n/2}.$ Then we have
    \[\mathbb{P}\paren{H\paren{\Lambda_{n/4,n}}} \geq \frac{1}{4}\mathbb{P}\paren{L_n}\,.\]
\end{lemma}

\begin{proof}
    Notice that an activated enhancement of radius $3n/2$ anchored in $\Lambda_{n/2}$ contains a path from $\partial \Lambda_{n/2}$ to $\partial \Lambda_n.$ The annulus $\Lambda_n \setminus \Lambda_{n/2}$ can be written as a union of four rotated and translated copies of $\Lambda_{n/4,n},$ and such a path must cross one of these rectangles in the short direction. Using a union bound then gives the desired inequality.
\end{proof}

We then use the following result of K\"{o}hler-Schindler and Tassion to turn a positive probability of a crossing of a rectangle the short way into a positive probability of crossing a rectangle the long way~\cite{kohler2023crossing}.

\begin{theorem}[K\"{o}hler-Schindler, Tassion]\label{thm:shortlongcrossing}
    For every $\rho \geq 1,$ there exists a homeomorphism $\psi_{\rho} : \brac{0,1} \to \brac{0,1}$ such that for every positively associated measure $\mathbb{P}$ which is invariant under the symmetries of $\R^2$ and every $n \geq 1,$
    \[\mathbb{P}\paren{H\paren{\Lambda_{n,\rho n}}} \geq \psi_{\rho}\paren{\mathbb{P}\paren{H\paren{\Lambda_{\rho n, n}}}}\]
\end{theorem}

\begin{prop}\label{prop:supercrossing}
    Let $\mathscr{E}$ be an enhancement family satisfying the hypotheses of Theorem~\ref{theorem:z2continuity} and let $p > q_c.$ Then for any $\epsilon > 0$ there is an $k$ so that 
    \[\mathbb{P}_p\paren{H\paren{\Lambda_{6*2^k,2*2^k},P_{k-1}}} \geq 1-\epsilon\,.\]
\end{prop}

\begin{proof}
    By Lemma~\ref{lemma:occupancycrossing} and Theorem~\ref{thm:shortlongcrossing}, $\mathbb{P}_p\paren{H\paren{\Lambda_{6*2^k,2*2^k},P_{k-1}}}$ is bounded below by a function of $\mathbb{P}_p\paren{G}.$ Then using  Lemma~\ref{lemma:torussharp} and Corollary~\ref{cor:summable} we obtain the desired lower bound.
\end{proof}

\begin{proof}[Proof of Theorem~\ref{theorem:qccontinuity} in the case $p>q_c$]
    For any $p>q_c,$ Proposition~\ref{prop:supercrossing} gives the same bound on the probability of box crossings as in Equation~(\ref{eq:boxcrossing}). This is sufficient to perform the same renormalization construction, so the truncated system must percolate.
\end{proof}

\section{Continuity in Higher Dimensions}

None of the previous arguments work in higher dimensions because we lose the ability to combine overlapping paths. As a result, Question~\ref{question:cont} is more difficult to approach in that setting. However, we are able to show that truncations of sufficiently thick enhancement families percolate above $q_c.$

We say that the enhancement family $\mathscr{E}$ is \textit{rotund} if there is an $c>0$ so that for every $\paren{T_\alpha,S_{\alpha}} \in \mathscr{E},$ $\Lambda_{\lfloor c r_\alpha \rfloor} \subset S_{\alpha},$ where $r_\alpha$ is the radius of $\paren{T_\alpha,S_{\alpha}}.$

\begin{theorem}\label{theorem:diskcontinuity}
    Let $d \geq 2$ and let $\mathscr{E}$ be a rotund family of enhancements. Then for any $p > q_c$ there is an $N \in \N$ so that $P_N$ percolates.    
\end{theorem}

In order to show that a finite set of enhancements are sufficient to percolate above $q_c,$ we now compare our enhanced set to a Poisson multiscale process studied in~\cite{menshikov2001connectivity}. As a first step, we relate high density disk processes with limited spatial dependence to Poisson disk processes, a continuous analogue of the stochastic domination result of~\cite{liggett1997domination}. In analogy to the discrete setting, we say that a point process $X$ on $\R^d$ is $k$-dependent if for any $A,B \subset \R^d$ with $d\paren{A,B} \geq k,$ $X\mid_A$ and $X\mid_B$ are independent. We also write $D\paren{X,r} \coloneqq \bigcup_{x \in X} B_r\paren{x}$ for the disk process of radius $r$ centered at the points of $X.$

\begin{lemma}\label{lemma:continuousLSS}
    For every $\lambda \in \R^+, k \in \N$ there is a $\delta>0$ so that if $X$ is a $k$-dependent point process on $\R^d$ which is invariant with respect to translations by $\Z^d$ with 
    \[\mathbb{P}_X\paren{\Lambda_{1/2}\text{ is occupied}} > 1-\delta\,\]
    there is a coupling of $X$ and a Poisson point process $Y$ with intensity $\lambda$ so that 
    \[D\paren{Y,1} \subset D\paren{X,1}\]
    almost surely.
\end{lemma}

\begin{proof}
    Our proof strategy will be to make use of a chain of stochastic dominations between $D\paren{X,1}$ and $D\paren{Y,1}$ with discretized intermediate steps. However, stochastic domination is not a transitive relation in general so we must take some additional care in constructing the final coupling.

    Consider a tiling of $\R^d$ by translations of $\Lambda_{1/2}.$ Say that one of the tiles is occupied if it contains a point of $X,$ and let $W$ be the union of the occupied tiles. By assumption, each tile is occupied with probability at least $1-\delta,$ and the states of tiles are $k$-dependent for some $k \in \N.$ Now let $U$ be a union of a $\mathrm{Bernoulli}\paren{1-e^{-\lambda}}$ set of tiles and let $U'$ be the union of tiles which have nonempty intersection with $U.$ Then by a slight modification of the argument in~\cite{liggett1997domination}, if $\delta$ is sufficiently small, $W$ stochastically dominates $U'.$ Since $W$ is a deterministic function of $X,$ it follows that there is a coupling between $X$ and $U$ so that $U' \subset D\paren{X,1}$ almost surely.

    Now for each tile $T,$ let $Y_T$ be a Poisson point process with intensity $\lambda$ on $T$ conditioned on having at least one point if $T \subset U$ and empty otherwise. Then $Y \coloneqq \bigcup_T Y_T$ is distributed as a Poisson point process with intensity $\lambda$ on $\R^d$ and $D\paren{Y,1} \subset U' \subset D\paren{X,1}$ almost surely as desired.
\end{proof}

\begin{proof}[Proof of Theorem~\ref{theorem:diskcontinuity}]
    Let $\lambda_c$ be the critical intensity for percolation in the Poisson disk process in $\R^d.$ Let $\lambda > \lambda_c$ and let $Z\paren{N}$ be a Poisson point process of intensity $\frac{\lambda}{N^d}.$ Since $p > q_c,$ by Corollary~\ref{cor:summable} and Lemma~\ref{lemma:continuousLSS}, there is an $N \in \N$ so that $P_N$ almost surely contains the set of edges in $D\paren{Z\paren{N},2N}.$ Thus, the connected component of the origin in $P_N$ almost surely contains the sites in the connected component of the origin in $D\paren{Z\paren{N},N},$ and since the latter is supercritical by construction, so is $P_N.$
\end{proof}

We mention in passing that in order to show continuity below $q_c,$ it would be enough to prove the following conjecture.

\begin{conjecture}
    There is an $\epsilon>0$ and an $N \in \N$ so that if $\mathbb{P}$ is an automorphism invariant positively associated percolation measure on $\Z^d$ with 
    \[\mathbb{P}\paren{H\paren{\Lambda_{2N,N,N}}} \geq 1-\epsilon\,,\]
    then 
    \[\mathbb{P}\paren{0 \leftrightarrow \infty} > 0\,.\]
\end{conjecture}

However, this would amount to a finite size criterion, which is known to be at least as difficult as major open problems in percolation~\cite{kozma2024reduction}.

\bibliographystyle{plain}
\bibliography{refs}
\end{document}